\documentclass{amsart}
\usepackage{hyperref}
\usepackage{enumerate}
\usepackage{a4wide}
\usepackage{amsrefs}
\usepackage{graphicx}
\usepackage{pict2e}

\theoremstyle{plain}
\newtheorem{claim}{Claim}
\newtheorem{theorem}{Theorem}
\newtheorem{proposition}{Proposition}
\newtheorem{lemma}{Lemma}
\newtheorem{fact}{Fact}
\newtheorem{corollary}{Corollary}

\theoremstyle{definition}
\newtheorem{definition}{Definition}
\newtheorem{example}{Example}
\newtheorem{notation}{Notation}

\theoremstyle{remark}
\newtheorem{remark}{Remark}

\newcommand{\rot}{\operatorname{rot}}
\newcommand{\LCS}{\operatorname{LCS}}

\newcommand{\selftanM}{
\begin{picture}(20,12)
\put(4,0){\line(0,1){13}}
\put(8.5,0){\line(0,1){13}}
\put(14,4){$\leftrightarrow$}
\end{picture}
\begin{picture}(20,12)
\qbezier(9.5,10)(10.5,11.5)(10.5,13)
\qbezier(9.5,3)(7.5,7)(9.5,10)
\qbezier(9.5,3)(10.5,1.5)(10.5,0)
\qbezier(7.5,10)(6.5,11.5)(6.5,13)
\qbezier(7.5,3)(9.5,7)(7.5,10)
\qbezier(7.5,3)(6.5,1.5)(6.5,0)
\put(14,4){$\leftrightarrow$}
\end{picture}
\begin{picture}(20,12)
\qbezier(8.5,10)(10.5,11.5)(10.5,13)
\qbezier(8.5,3)(4,7)(8.5,10)
\qbezier(8.5,3)(10.5,1.5)(10.5,0)
\qbezier(8.5,10)(6.5,11.5)(6.5,13)
\qbezier(8.5,3)(13,7)(8.5,10)
\qbezier(8.5,3)(6.5,1.5)(6.5,0)
\end{picture}
}
\newcommand{\selftanMa}{
\begin{picture}(20,12)
\put(4,-2){\line(0,1){13}}
\put(8.5,-2){\line(0,1){13}}
\put(14,2){$\leftrightarrow$}
\end{picture}
\begin{picture}(20,12)
\qbezier(8.5,8)(10.5,9.5)(10.5,11)
\qbezier(8.5,1)(4,5)(8.5,8)
\qbezier(8.5,1)(10.5,-0.5)(10.5,-2)
\qbezier(8.5,8)(6.5,9.5)(6.5,11)
\qbezier(8.5,1)(13,5)(8.5,8)
\qbezier(8.5,1)(6.5,-0.5)(6.5,-2)
\end{picture}
}

\newcommand{\selftanMb}{
\begin{picture}(20,15)
\qbezier(8.5,10)(10.5,11.5)(10.5,13)
\qbezier(8.5,3)(4,7)(8.5,10)
\qbezier(8.5,3)(10.5,1.5)(10.5,0)
\qbezier(8.5,10)(6.5,11.5)(6.5,13)
\qbezier(8.5,3)(13,7)(8.5,10)
\qbezier(8.5,3)(6.5,1.5)(6.5,0)
\put(14,4){$\rightarrow$}
\end{picture}
\begin{picture}(20,15)
\put(4,0){\line(0,1){13}}
\put(8.5,0){\line(0,1){13}}
\end{picture}
}

\newcommand{\tripleM}{
\begin{picture}(30,0)
\qbezier(4.5,10)(6.5,11.5)(6.5,13)
\qbezier(4.5,3)(0,7)(4.5,10)
\qbezier(4.5,3)(6.5,1.5)(6.5,0)
\put(0,0){\line(1,1){13}}
\put(13,0){\line(-1,1){13}}
\put(17,4){$\leftrightarrow$}
\end{picture}
\begin{picture}(30,0)
\put(6.5,0){\line(0,1){13}}
\put(0,0){\line(1,1){13}}
\put(13,0){\line(-1,1){13}}
\put(17,4){$\leftrightarrow$}
\end{picture}
\begin{picture}(30,0)
\qbezier(8.5,10)(6.5,11.5)(6.5,13)
\qbezier(8.5,3)(13,7)(8.5,10)
\qbezier(8.5,3)(6.5,1.5)(6.5,0)
\put(0,0){\line(1,1){13}}
\put(13,0){\line(-1,1){13}}
\end{picture}
}

\begin{document}
\title[Commutators of pure twin groups and rondles]{Commutators of pure twin groups and rondles}
\author{Noboru Ito}
\address{Department of Mathematics, Faculty of Engineering, Shinshu University, Wakasato 4-17-1, Nagano, Nagano 380-8553, Japan}
\email{nito@shinshu-u.ac.jp}
\keywords{twin group; pure twin group; plane curve; rondle; finite-order invariant; triple-point modification; lower central series}
\date{May 24, 2026}
\maketitle

\begin{abstract}
We formulate a relationship between finite-order rondle invariants with respect to triple-point modifications and the lower central series of subgroups of a pure twin group.   
Using our formulation, we construct infinitely many infinite sequences of prime plane curves such that, for each sequence, the curves share the same invariants up to any fixed order.
\end{abstract}
\section{Introduction}\label{sec:intro}
A \emph{long curve} is a generic immersion of an oriented line into $\mathbb{R}^2$ that coincides with the $x$-axis outside a certain disc.  Two long curves are equivalent if they are obtained from each other by a finite sequence of two types of moves.  The first is a diffeomorphism of $\mathbb{R}^2$ with compact support, and the second is a self-tangency modification, i.e.,  \selftanMa. 
Every long curve is associated with a closed curve on $\mathbb{R}^2 \cup \{ \infty \}$ with the base point $\{ \infty \}$.   Since an equivalence class of one-component closed curves under the two moves described above is known as a one-component \emph{rondle}, we freely identify a long curve with a one-component rondle having a base point.  
Extending the definition of a one-component rondle without a base point to the multi-component case, also called a \emph{rondle}, is straightforward.  In this paper, rondles and long curves are considered in the plane unless stated.   
\begin{theorem}\label{thm:mainLCS}
Let $C$ and $C'$ be two rondles that differ by a twin $p \in \LCS_n (\overline{PTW_k})$, the $n$th group of the lower central series of $\overline{PTW_k}$, the small pure twin group on $k$ arcs.  Let $f$ be a rondle  invariant of order less than $n$ with respect to triple-point modifications.  Then $f (C)= f (C')$.      
\end{theorem}
\noindent Precise definitions of the terms are given in Section~\ref{sec:prelim}.  The \emph{twin group} $TW_k$ on $k$ arcs is generated by a set of $(k-1)$ generators   $\{\sigma_i \mid i=1, 2, \ldots, k-1\}$ satisfying the following relations: 
\begin{align}
\sigma^2_i &=1 \quad (\forall i), \label{eq:in} \\
\sigma_i \sigma_j &= \sigma_j \sigma_i  \quad (|i-j|>1).  \label{eq:slide}
\end{align}
The \emph{pure twin group} $PTW_k$ on $k$ arcs is the kernel of a natural homomorphism $TW_k$ $\to$ $S_k$.  The subgroup $\overline{PTW_k}$ of $PTW_k$ is called a \emph{small pure twin group} on $k$ arcs if it    
is generated by elements  
$p^{(1)}, p^{(2)}, \dots,  p^{(\ell)}$, each of which is reduced to $1$ by applying    
\begin{equation}\label{eq:YB}
\sigma_i \sigma_{i+1} \sigma_i = \sigma_{i+1} \sigma_i \sigma_{i+1} \qquad (1 \le i \le k-2) \qquad {\textrm{(Yang-Baxter relation)}}, 
\end{equation}
exactly once, with relations (\ref{eq:in}) and (\ref{eq:slide}) freely used.   We will prove  Theorem~\ref{thm:mainLCS} in Section~\ref{sec:proof:mainLCS}.   

In Section~\ref{sec:prelim}, we define the order with respect to triple-point modification for a rondle invariant, and also define what it means for two rondles  to \emph{differ by} a twin $p$.  As a typical example, if a rondle $\bar{x}$ represents the closure of a twin $x$, and $p$ and $b$ are any two twins with the same number of arcs, then $\bar{b}$ and $\overline{pb}$ \emph{differ by} $p$.  
Theorem~\ref{thm:mainInf} provides a method for modifying rondles without changing their invariants up to some order.
Such modifications are related to algebraic structures arising from the topology of the space of immersions and the strata corresponding to triple intersections.
\begin{theorem}\label{thm:mainInf}
For any positive integers $k$ and $n$ with $k \ge 4$, there exist infinitely many pairs $C_i, C'_i$ of prime rondles such that $f(C_i)=f(C'_i)$ $(i=1, 2, 3, \dots)$ for any rondle invariant $f$ of order less than $n$.  
Further, for each $i$, there exists a sequence $C'_i=C^{(1)}_i, C^{(2)}_i, C^{(3)}_i \dots$ such that $f(C_i)=f(C^{(j)}_i)$ $(j=1, 2, 3, \dots)$.   
\end{theorem} 
\noindent We will prove Theorem~\ref{thm:mainInf} in Section~\ref{sec:construction}, and we remark that Fact~\ref{factRondle} ensures that Theorem~\ref{thm:mainInf} applies to all rondles on $S^2$:  
\begin{fact}[\cite{Khovanov1996}]\label{factRondle}
Every rondle on $S^2$ is the closure of a twin.  
\end{fact}
A closed curve $C$ in $\mathbb{R}^2$, with or without a base point, is said to be  \emph{composite} if there exists a circle $S$ embedded in $\mathbb{R}^2$ which intersects $C$ in exactly two points, and such that $S$ does not bound a $2$-disk $d$ with $d \cap C$ isotopic inside $d$ to an arc with no double points.  A closed curve $C$ is \emph{prime} if it is not composite.    A \emph{prime rondle} is a rondle that has a  representative which is a prime plane curve.    

A rondle invariant of order zero with respect to triple-point modifications is obtained from  
the rotation number, defined as the integral of the curvature with respect to the arc-length parameter, which was introduced by Hopf \cite{Hopf1935} and later developed by Whitney \cite{Whitney1937}.  Order one invariants are called \emph{Arnold invariants} $J^+$, $J^-$, and $St$ \cite{Arnold1994Book}.    
Higher order $J^+$ invariants have been well studied in the context of Legendrian knots (e.g.,  \cite{Goryunov1998}, \cite{Tchernov2002}), whereas higher-order $J^-$ invariants have been characterized by the geometry and topology of complexification, triggered by Rokhlin formula \cite{Viro1996}.  On the other hand, $St$-theory has developed in a relatively independent manner.   A direction of higher-order $St$ invariants is to give concrete partition functions \cite{Shumakovich1995, Tabachnikov1996}  and the other is an analogue of the Vassiliev skein relation \cite{ArakawaOzawa1999}.   For a plane curve $C$ and a parameter $q$,  quantized polynomial invariants $P_C (q)$, $I_q(C)$, and $St_q (C)$ were introduced in  \cite{Viro1996}, \cite{LanzatPolyak2013},  and \cite{Ito2023}, respectively, corresponding to $J^-$, $J^+$, and $St$.   For $P_C (q)$ of a $k$-component plane curve $C$, by substituting $q=e^x$ into them, the first  coefficient is $\rot(C)$ and the second one is $k-J^-$.  For $I_q (C)$ ($St_q$,~resp.), by the Taylor expansion at $q=1$, the zeroth coefficient is $\rot(C)$ and the first one $I'_1 (C)$ ($St'_1$,~resp.) is $\frac{1-J^+}{2}$ ($St$,~resp.).  
The coefficients of this expansion provide a family of invariants related to the higher-order theory of Arnold's strangeness.
In particular, certain coefficients $St^r(C)$ are known as Tabachnikov invariants.

Before closing this introduction, we briefly outline  the structure of the paper.  Section~\ref{sec:prelim} reviews the definitions and notations used throughout.  
Section~\ref{sec:proof:mainLCS} contains the proof of the first main result (Theorem~\ref{thm:mainLCS}), and     Section~\ref{sec:DiagrammaticProof} presents a diagrammatic alternative proof of the same result (Theorem~\ref{thm:mainLCS}).      
Section~\ref{sec:construction} contains the proof of the second main result (Theorem~\ref{thm:mainInf}).  
Section~\ref{sec:appendix} provides a proof of a key claim used in the argument for the  second main result (Theorem~\ref{thm:mainInf}), following the approach of \cite{Khovanov1997}.  Readers familiar with \cite{Khovanov1996, Stanford1996} may find Sections~\ref{sec:prelim} and \ref{sec:appendix} somewhat more accessible; 
\cite{Stanford1996} is relevant to Section~\ref{sec:proof:mainLCS}, and \cite{Ohyama1995} relates to Section~\ref{sec:DiagrammaticProof}.   
Except for an elementary inductive lemma (Lemma~\ref{lem:Kformula}), which is similar to one in \cite{Ohyama1995}, the exposition is essentially self-contained and may be read independently of these references.     

\section{Definitions and notations}\label{sec:prelim}
\begin{definition}[twin group]\label{def:twin}
The \emph{twin group} $TW_k$ on $k$ arcs is generated by a set of $(k-1)$ generators   $\{\sigma_i \mid i=1, 2, \ldots, k-1\}$ that satisfy relations (\ref{eq:in}) and (\ref{eq:slide}) in  Section~\ref{sec:intro}.  
\end{definition}
\begin{definition}[pure twin group]\label{def:puretwin}
The \emph{pure twin group} $PTW_k$ on $k$ arcs is the kernel of a natural homomorphism $TW_k$ to $S_k$, which is the symmetric group of degree $k$.    
\end{definition}
\begin{definition}[small pure twin group]\label{def:smallpuretwin}
A subgroup $\overline{PTW_k}$ of $PTW_k$ is called a \emph{small pure twin group} on $k$ arcs if it is generated by elements
$p^{(1)}, p^{(2)}, \dots, p^{(\ell)} \in PTW_k$,
each of which is reduced to $1$ by applying (\ref{eq:YB}) exactly once, with relations (\ref{eq:in}) and (\ref{eq:slide}) freely used.  
\end{definition}
\begin{definition}[configuration]\label{def:configuration}
Let $\mathbb{R}^2 = \{(x, y) \mid x, y \in \mathbb{R} \}$.  We take $n$ points $(1, 0)$, $(2, 0), \dots$, $(n,0)$ on the line $y=0$ and also $n$ points $(1, 1)$, $(2, 1), \dots, (n,1)$ on the line $y=1$.  We consider configurations of $n$ arcs in $\mathbb{R} \times [0, 1]$ that 
connect the points on $y=0$ and those on $y=1$ by $n$ arcs.    
Then we require the conditions (\ref{CondiMono}) and (\ref{CondiTriple}): 
\begin{enumerate}[(i)]
\item \label{CondiMono}
the projection of any arc to the $y$-coordinate is a homeomorphism onto $[0, 1]$; that is, every arc is monotonic.   
\item \label{CondiTriple}
there are no three arcs sharing a common point.  
\end{enumerate}
We say that two configurations satisfying (\ref{CondiMono}) and (\ref{CondiTriple}) are \emph{equivalent} if these are connected by a homotopy of arcs in $\mathbb{R} \times [0, 1]$ such that conditions (\ref{CondiMono}) and (\ref{CondiTriple}) hold and the ends of arcs are fixed throughout the homotopy. 
The product of two configurations on the same number of arcs is defined by putting one on top of the other.    
\end{definition}
The following proposition is known (e.g., \cite{Khovanov1997}).  
\begin{proposition}\label{prop:GeomRealization}
If configurations with $n$ arcs satisfy (\ref{CondiMono}) and (\ref{CondiTriple}), their set forms a twin group $TW_n$.   
\end{proposition}
By Proposition~\ref{prop:GeomRealization}, relations of twins explicitly correspond to local transformations of rondles as follows.        
\begin{definition}[self-tangency modification]
A \emph{self-tangency modification} is a local transformation $\selftanM$ that corresponds to the relation (\ref{eq:in}).  
\end{definition}
\begin{definition}[triple-point modification]
A \emph{triple-point modification} is a local transformation  
\begin{picture}(75,0)
\put(0,-5){$\tripleM$} 
\end{picture}
which corresponds to relation (\ref{eq:YB}).  
\end{definition}
\begin{definition}[rondle invariant, long curve invariant]
A map from the set of rondles (long curves,~resp.) to a set is called a rondle (long curve,~resp.) invariant.  In particular, any rondle or long curve invariant is unchanged by self-tangency modifications, where any self-tangency modification occurs away from the base point ($\infty$,~resp.) for rondles with base points (long curves,~resp.).  
\end{definition}
\begin{definition}[disjoint triangles]\label{def:DisjointTriangle}
A set $s$ of triangles is said to be \emph{disjoint} if any two triangles in $s$ share no vertex, edge, or face.   
\end{definition}
\begin{definition}[finite orders of long curve/rondle invariants]
\label{def:FiniteOrder}
A long curve (rondle,~resp.)  invariant $r$ taking values in an abelian group is said to have \emph{order} $n-1$, with respect to triple-point modifications, if there exists a positive  integer $n$ such that for any long curve (rondle,~resp.) with $n$ chosen disjoint triangles numbered $1$ to $n$
\begin{equation}\label{eq:VanishingSum}
\sum_{s \subset \{1, 2, \ldots n \}} (-1)^{|s|} r (D_{s}) = 0
\end{equation}
where $|s|$ is the cardinality of a subset $s$ of $\{1, 2, \dots n \}$, and $D_{s}$ is obtained from $D=D_{\emptyset}$ by applying the triple-point modification to each triangle in $s$.   The smallest such $n$ is called the order of $r$, and if there is no such $n$, then we say that $r$ has \emph{infinite order}.  
\end{definition}
By Definition~\ref{def:smallpuretwin} and Proposition~\ref{prop:GeomRealization}, we have: 
\begin{proposition}\label{prop:smallpuretwin}
Any configuration corresponding to $p \in \overline{PTW_k}$ is deformed into $1$ by a series of triple-point modifications.  
\end{proposition}
An outline of proof of  Proposition~\ref{prop:GeomRealization} is given by an identification between $\sigma_i$ and a configuration that has  exactly one double point consisting of the $i$th and $(i+1)$th arcs.  Details are left to the reader; see \cite[Section~1]{Khovanov1997}.   
In this paper, we freely use this identification between twins of $TW_k$ and configurations with $k$ arcs that satisfy (\ref{CondiMono}) and (\ref{CondiTriple}).   
\begin{definition}[tangle map (cf.~\cite{Stanford1996})]
The image of a generic immersion $S^1 \sqcup S^1 \sqcup \cdots \sqcup S^1 \to \mathbb{R}^2$ such that all intersections are transverse double points is a \emph{link projection} and each image of a single $S^1$ is a \emph{strand}.   
A \emph{tangle projection} is a link projection with one boundary component homeomorphic to $S^1$, in which some strands may intersect transversely.  
By a \emph{tangle map} $T : TW_k \to \{\text{rondles}\}$ ($\{\text{long curves}\}$,~resp.) we mean a fixed way of putting a twin $p \in TW_k$ into a tangle projection to obtain a rondle (long curve,~resp.)  $T(p)$ as an equivalence class of link projections (single-component link projections with a fixed base point,~resp.).  
This map is well-defined since the defining relations of the twin group correspond to self-tangency modifications of link projections up to compactly supported diffeomorphisms of $\mathbb{R}^2$.
In particular, $T(p)$ is called the {\emph{twin closure}} or the {\emph{closure}} $\bar{p}$ if the $i$th strand on the bottom is looped around and identified with the $i$th strand on the top.  
For example, see Figure~\ref{fig:Tmap}.    
\end{definition}
\begin{figure}[htbp] 
   \centering
\includegraphics[width=10cm]{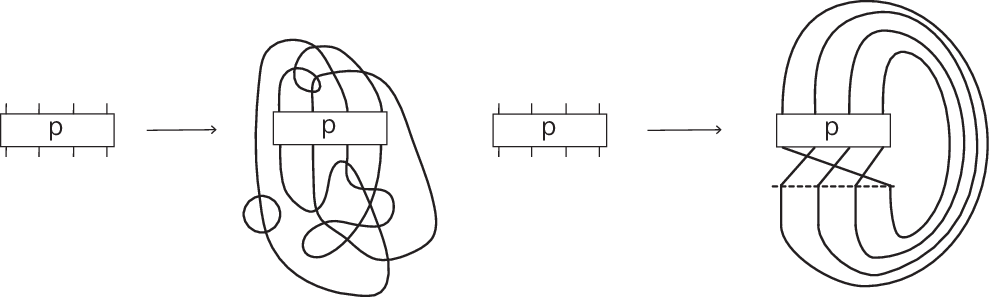} 
   \caption{Two examples of a tangle map $T(p)$ of a pure twin $p$ ($\in PTW_4$), yielding a multi-component rondle (left) or a one-component rondle (right).  In particular, the closure of $p \sigma_1 \sigma_2 \sigma_3$ is always a one-component curve (right).}
   \label{fig:Tmap}
\end{figure}
\begin{definition}[differ by a twin]
Let $x$ be a twin.    Two plane curves $C$ and $C'$ are said to \emph{differ by} a twin $p$ if $C= T(x)$ and $C' = T(px)$, where  $T(x)$ stands for the image of a tangle map $T$ of $x$.   
\end{definition}
By the definition of ``differ by'', there exists  a twin $x$ and a tangle map $T$ such that $T(x)=C$ and $T(px)=C'$.  
Thus, without loss of generality, we also use the following notation.   
\begin{notation}[$T_x$]
For a fixed twin $x$, define
$
T_x(q)=T(qx).
$
Then $T_x$ is also a tangle map, and
$
T_x(1)=C$, $T_x(p)=C'.
$
\end{notation}
\begin{definition}[ideals $J$, $J^n$]
Let $J \subset \mathbb{Q} TW_k$ denote the ideal generated by 
\[\{ 
(\sigma_i \sigma_{i+1} \sigma_i - \sigma_{i+1} \sigma_i \sigma_{i+1}) : 1 \le i \le k-2 \}.\] 
For $n \in \mathbb{Z}_{>0}$, let $J^n$ be the two-sided ideal generated by any product $x_1 x_2 \cdots x_n$ for each $x_i \in J$.
\end{definition}
\begin{definition}[$\LCS (G)$ of a group $G$]
\label{def:LCS}
Given a group $G$, its lower central series $\LCS_1(G)$, $\LCS_2(G)$, $\LCS_3(G), \cdots$ is defined inductively by $\LCS_{n+1}(G)$ $=$ $[G, \LCS_n(G)]$ and $\LCS_1(G)=G$, where $[H, K]$ $=$ $\langle xyx^{-1}y^{-1} : x \in H~{\textrm{and}}~y \in K \rangle$.
\end{definition}
\begin{notation}[$f_{r, T}$]
Given a tangle map $T$ and a rondle invariant $r$, $f_{r, T}$ denotes the linear extension of $r \circ T : TW_k \to \mathbb{Q}$ to the group algebra $\mathbb{Q} TW_k$.    
\end{notation}
\section{Proof of Theorem~\ref{thm:mainLCS}}\label{sec:proof:mainLCS}
\begin{lemma}\label{lem:twinSt}
Let $r$ be a rondle invariant of order $n-1$, and let $T$ be a tangle map.    
If $x \in J^{n}$, then $f_{r, T} (x)=0$.  
\end{lemma}
\begin{proof}
The ideal $J^{n}$ is generated by expressions of the following form: 
\begin{equation}\label{eq:altForm}
w_0(\sigma_{m_1} \sigma_{m_1+1} \sigma_{m_1} - \sigma_{m_1+1} \sigma_{m_1} \sigma_{m_1+1})w_1\cdots w_{n-1}(\sigma_{m_n} \sigma_{m_{n} +1} \sigma_{m_{n}} - \sigma_{m_{n} +1} \sigma_{m_{n}} \sigma_{m_{n}+1})w_{n}
\end{equation}
where each $w_j$ is an element of $TW_k$ and each $\sigma_{m_j}$ denotes a standard twin generator corresponding to a single double point.       
For an element of the form (\ref{eq:altForm}), consider 
a twin $p_0$ of the form: \[w_0 \tau_{m_1} w_1 \cdots w_{n-1} \tau_{m_{n}} w_{n}\]
where each $\tau_{m_j}$ denotes $\sigma_{m_j} \sigma_{m_j +1} \sigma_{m_j}$ or $\sigma_{m_j + 1} \sigma_{m_j} \sigma_{m_j + 1}$ and each $\sigma_{m_j}$ denotes a standard twin generator corresponding to a single double point, without cancellation with $w_{j-1}$ or $w_{j}$.  Let $t_j$ denote a triangle corresponding to $\tau_{m_j}$.   Note that any two triangles in $\{t_1, t_2, \dots, t_{n} \}$ share no  vertex, edge, or face.  Let $D$ be a link  projection corresponding to $T(p_0)$, and for each subset $s \subset \{1, 2, \dots, n \}$, let $D_s$ be the link projection obtained from $D$ by applying triple-point modifications to triangles corresponding to the label $t_j$ ($j \in s$).  Note that $T(p_0)$ corresponds to $D$ ($=D_\emptyset$).  Then 
expanding each factor produces the alternating sum over all subsets $s \subset \{1, 2, \ldots , n \}$,  
\begin{equation}\label{eq:FtoSum}
f_{r, T}(x) = \sum_{s \subset \{1, 2, \ldots, n \}} (-1)^{|s|} r (D_{s}).  
\end{equation}
The right-hand side of (\ref{eq:FtoSum}) vanishes by (\ref{eq:VanishingSum}) in  Definition~\ref{def:FiniteOrder}.  
\end{proof}   
\begin{lemma}\label{lem:Pone}
If $p \in {\LCS}_n (\overline{PTW_k})$,  then $p-1 \in J^n$.  
\end{lemma}
\begin{proof}
We prove the claim by induction on $n$.   
\begin{enumerate}
\item Case~$n=1$.   If $p \in {\LCS}_1 (\overline{PTW_k})$ $=$ $\overline{PTW_k}$, then $p$ is deformed into $1$ by a series of triple-point modifications (Proposition~\ref{prop:smallpuretwin}).   Then
\begin{equation}\label{eq:InductionFirst}
p-1 \in J.  
\end{equation}
\item Case~$n$.  Suppose that the claim for $n-1$ holds, i.e., 
\begin{equation}\label{eq:Induction}
p \in \LCS_{n-1}(\overline{PTW_k})
\Rightarrow p-1 \in J^{n-1}.  
\end{equation} 
If $p \in {\LCS}_n  (\overline{PTW_k})$, then 
\[
p = [p_1, q_1][p_2, q_2]\cdots [p_m, q_m], 
\]
where $p_i \in \overline{PTW_k}$ and $q_i \in \LCS_{n-1}(\overline{PTW_k})$.   
Therefore it suffices to show that each $[p_i,q_i]-1$ belongs to $J^n$.  
Here, recall that, in general, the equation 
\[
g_1 g_2 \cdots g_m -1 = \sum_{i=1}^{m} g_1 \cdots g_{i-1} (g_i -1)
\]
holds.  
Next, let $a=p_i$ and $b=q_i$.  
For any $a, b \in TW_k$, the following identity holds:  
\[
[a, b] -1 = (ab - ba) a^{-1} b^{-1} = (
 (a -1)(b -1) - (b -1) (a -1)) a^{-1} b^{-1}.  
\] 
Thus $[p_i, q_i]-1 \in J^n$, since (\ref{eq:InductionFirst}) of Case~$n=1$ gives  $p_i -1 \in J$, and (\ref{eq:Induction}) implies  $q_i -1 \in J^{n-1}$.  
\end{enumerate}
\end{proof}
\noindent \emph{Proof of Theorem~\ref{thm:mainLCS}.}  
By the definition of ``differ by'', there exists a tangle map $T_x$ such that $T_x(1)=C$ and $T_x(p)=C'$.    
By Lemmas~\ref{lem:twinSt} and \ref{lem:Pone}, we have $f_{r, T_x} (p) - f_{r, T_x} (1)=f_{r, T_x} (p-1)=0$.  
\hfill$\Box$
\section{Diagrammatic proof of Theorem~\ref{thm:mainLCS}}\label{sec:DiagrammaticProof}
\begin{definition}[triangle]
Let $C$ be a representative of a twin.   If there exists a disk in $C$ that looks like the one on the left or the right in Figure~\ref{fig:triaglePair}, then we call the disk a \emph{triangle}.  For a certain $i$, a triangle corresponding to $\sigma_i \sigma_{i+1} \sigma_i$ is called positive (Figure~\ref{fig:triaglePair}, left), and a triangle corresponding to $\sigma_{i+1} \sigma_i \sigma_{i+1}$ is called negative  (Figure~\ref{fig:triaglePair}, right).    
When specifying the sign $\varepsilon$ ($(-\varepsilon)$,~resp.), we call the triangle an $\varepsilon$-triangle ($(-\varepsilon)$-triangle,~resp.), where $\varepsilon=\pm$ corresponds to $(-\varepsilon)=\mp$, with the same order of signs preserved.     
\end{definition}
\begin{figure}[htbp] 
   \centering
   \includegraphics[width=5cm]{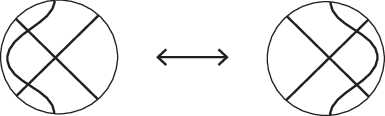} 
   \caption{A positive triangle (left) and a negative triangle (right).  They are exchanged by the triple-point modification.}
   \label{fig:triaglePair}
\end{figure}
\begin{definition}[$T_n$-similar, $T_n$-trivial]
Let $K$ ($L$,~resp.) be a twin and let $\widetilde{K}$ ($\widetilde{L}$,~resp.) be a representative of $K$ ($L$,~resp.).  Suppose that $A_1$, $A_2$, \dots, $A_n$ are non-empty sets of disjoint triangles in $\widetilde{K}$ or $\widetilde{L}$.  We say that $K$ and $L$ are $T_n$-\emph{similar} if there exist $A_1$, $A_2$, \dots, $A_n$ such that
\begin{itemize}
\item $A_i \cap A_j = \emptyset$ ($i \neq j$), 
\item Replacing each $\varepsilon$-triangle with a $(- \varepsilon)$-triangle at each  triangle in any non-empty subfamily of $A_1$, $A_2$, \dots, $A_n$ transforms $\widetilde{K}$ into $\widetilde{L}$, or vice versa.  
In particular, 
if $L$ has a representative $\widetilde{L}$ with no double points, then $K$ is called $T_n$-\emph{trivial}.  
\end{itemize}
\end{definition}
Example~\ref{egTsimilar} shows a case where the definition of $T_n$-similarity applies.  In particular, this example shows how the conditions on disjoint $n$ sets of triangles ensure that two rondles are $T_n$-similar.  
\begin{example}\label{egTsimilar}
In Figure~\ref{fig:GeneThreeTri},  we see  an example of a $T_3$-trivial twin with $A_1=\{ t_{11}, t_{12}, t_{13}, t_{14} \}$, $A_2=\{ t_{21}, t_{22}, t_{23}, t_{24} \}$, $A_3=\{ t_{31}, t_{32} \}$. 
\begin{figure}[htbp] 
   \centering
   \includegraphics[width=12cm]{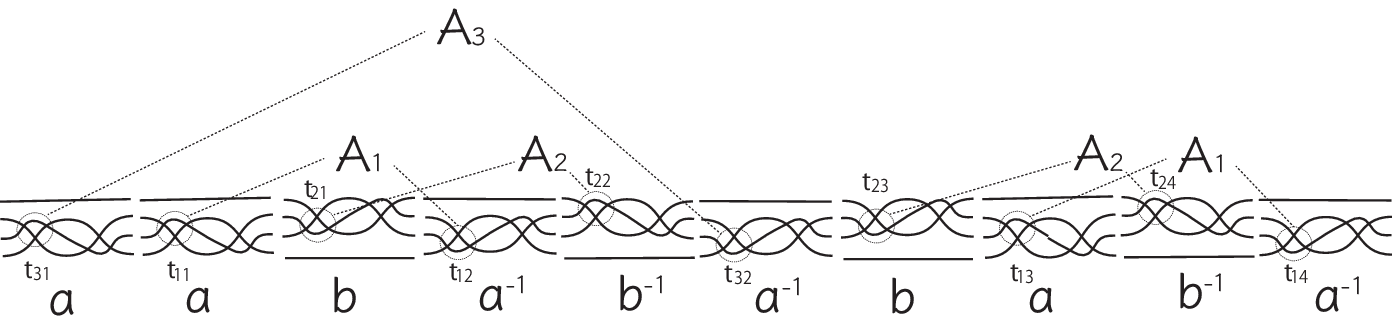} 
   \caption{An example of $T_3$-trivial twin.  This figure originally represented a vertical twin, rotated by $+ \frac{\pi}{2}$ to be drawn horizontally.}
   \label{fig:GeneThreeTri}
\end{figure} 
\end{example}
For every triple-point modification, we define the formal relation by 
\[
\begin{picture}(20,10)
\put(6.5,-4){\line(0,1){13}}
\put(0,-4){\line(1,1){13}}
\put(13,-4){\line(-1,1){13}}
\end{picture} := 
\begin{picture}(30,10)
\qbezier(4.5,6)(6.5,7.5)(6.5,9)
\qbezier(4.5,-1)(0,3)(4.5,6)
\qbezier(4.5,-1)(6.5,-2.5)(6.5,-4)
\put(0,-4){\line(1,1){13}}
\put(13,-4){\line(-1,1){13}}
\put(0,-8){\textrm{\tiny pos.}}
\put(20,0){-}
\end{picture} 
\begin{picture}(20,0)
\qbezier(8.5,6)(6.5,7.5)(6.5,9)
\qbezier(8.5,-1)(13,3)(8.5,6)
\qbezier(8.5,-1)(6.5,-3.5)(6.5,-4)
\put(0,-4){\line(1,1){13}}
\put(13,-4){\line(-1,1){13}}
\put(0,-8){\textrm{\tiny neg.}}
\end{picture}
.\]
It implies
\begin{align}
\label{eq:Trelations}
\begin{picture}(20,10)
\qbezier(4.5,6)(6.5,7.5)(6.5,9)
\qbezier(4.5,-1)(0,3)(4.5,6)
\qbezier(4.5,-1)(6.5,-2.5)(6.5,-4)
\put(0,-4){\line(1,1){13}}
\put(13,-4){\line(-1,1){13}}
\put(0,-8){\textrm{\tiny pos.}}
\end{picture} = 
\begin{picture}(20,0)
\qbezier(8.5,6)(6.5,7.5)(6.5,9)
\qbezier(8.5,-1)(13,3)(8.5,6)
\qbezier(8.5,-1)(6.5,-3.5)(6.5,-4)
\put(0,-4){\line(1,1){13}}
\put(13,-4){\line(-1,1){13}}
\put(0,-8){\textrm{\tiny neg.}}
\end{picture}
+ \begin{picture}(0,10)
\put(6.5,-4){\line(0,1){13}}
\put(0,-4){\line(1,1){13}}
\put(13,-4){\line(-1,1){13}}
\end{picture} \qquad, \quad
\begin{picture}(20,0)
\qbezier(8.5,6)(6.5,7.5)(6.5,9)
\qbezier(8.5,-1)(13,3)(8.5,6)
\qbezier(8.5,-1)(6.5,-3.5)(6.5,-4)
\put(0,-4){\line(1,1){13}}
\put(13,-4){\line(-1,1){13}}
\put(0,-8){\textrm{\tiny neg.}}
\end{picture} = 
\begin{picture}(20,10)
\qbezier(4.5,6)(6.5,7.5)(6.5,9)
\qbezier(4.5,-1)(0,3)(4.5,6)
\qbezier(4.5,-1)(6.5,-2.5)(6.5,-4)
\put(0,-4){\line(1,1){13}}
\put(13,-4){\line(-1,1){13}}
\put(0,-8){\textrm{\tiny pos.}}
\end{picture} 
- \begin{picture}(0,10)
\put(6.5,-4){\line(0,1){13}}
\put(0,-4){\line(1,1){13}}
\put(13,-4){\line(-1,1){13}}
\end{picture}\qquad.  
\end{align}
By applying these relations (\ref{eq:Trelations}) to a twin, a curve with  some triple points appears, which is called a \emph{singular twin}.    
Let $t_{i1}$ $t_{i2}, \dots, t_{i\alpha(i)}$ be disjoint $\varepsilon_{i1}$, $\varepsilon_{i2}, \dots, \varepsilon_{i\alpha(i)}$-triangles,   respectively, in a representative of a configuration and let $A_i = \{t_{i1}, t_{i2}, \dots, t_{i\alpha(i)} \}$.   
Let $K \left( \begin{matrix} 1&2& \dots & k \\ i_1& i_2 & \dots & i_k \end{matrix} \right)$ be a singular twin with triple points obtained by applying triple-point modifications at $t_{11}, t_{12}, \dots, t_{1 i_1 -1}$, $t_{21}, t_{22}, \dots, t_{2 i_2 -1}$, $t_{k1}, t_{k2}, \dots, t_{k i_k -1}$ to $K$, and then  
replacing triangles $t_{1 i_1}, t_{2 i_2}, \dots t_{k i_k}$ with triple points.  
\begin{example}
Let $\widetilde{K}$ be the representative of the twin $K$ as in Figure~\ref{fig:GeneThreeTri} and let $K \left( \begin{smallmatrix} 1 \\ \infty \end{smallmatrix} \right)$ be the twin  obtained by applying triple-point modifications to the triangles in $A_1$.   
In the following, $t_{1i}$ indicates the triple-point modification at $t_{1i}$.   A repeated application of (\ref{eq:Trelations}) yields:  

\begin{tabular}{ccccccccc}
$\widetilde{K}$ & $\stackrel{t_{11}}{\longrightarrow}$ & $\cdot$ & $\stackrel{t_{12}}{\longrightarrow}$ & $\cdot$ & $\stackrel{t_{13}}{\longrightarrow}$ & $\cdot$ & $\stackrel{t_{14}}{\longrightarrow}$ & $K \left( \begin{smallmatrix} 1 \\ \infty \end{smallmatrix} \right)$  \\
&$\searrow$&$-$&$\searrow$&$+$&$\searrow$&$-$& $\searrow$ & $+$
\\
&&$K \left( \begin{smallmatrix} 1 \\ 1 \end{smallmatrix} \right)$&&$K \left( \begin{smallmatrix} 1 \\ 2 \end{smallmatrix} \right)$&&$K \left( \begin{smallmatrix} 1 \\ 3 \end{smallmatrix} \right)$&\ \quad& $K \left( \begin{smallmatrix} 1 \\ 4 \end{smallmatrix} \right)$.   \\
\end{tabular}

\noindent For example, $K \left( \begin{smallmatrix} 1 \\ 3 \end{smallmatrix} \right)$ ($\widetilde{K \left( \begin{smallmatrix} 1 \\ 3 \end{smallmatrix} \right)}$,~resp.) is a singular twin (representative,~resp.) with a triple point by applying triple-point modifications at $t_{11}$  and $t_{12}$ and by replacing the triangle $t_{13}$ with the triple point.  For $K \left( \begin{smallmatrix} 1 \\ 3 \end{smallmatrix} \right)$, let $K \left( \begin{smallmatrix} 1& 2 \\ 3 & \infty \end{smallmatrix} \right)$ be the singular twin obtained by applying triple-point modifications to the triangles in $A_2$ included by $K \left( \begin{smallmatrix} 1 \\ 3 \end{smallmatrix} \right)$.   In the following diagram, $t_{2i}$ denotes the application of the triple-point modification at $t_{2i}$.  By repeatedly applying (\ref{eq:Trelations}) to the representative $\widetilde{K \left( \begin{smallmatrix} 1 \\ 3 \end{smallmatrix} \right)}$,

\begin{tabular}{ccccccccc}
$\widetilde{K \left( \begin{smallmatrix} 1 \\ 3 \end{smallmatrix} \right)}$ & $\stackrel{t_{21}}{\longrightarrow}$ & $\cdot$ & $\stackrel{t_{22}}{\longrightarrow}$ & $\cdot$ & $\stackrel{t_{23}}{\longrightarrow}$ & $\cdot$ & $\stackrel{t_{24}}{\longrightarrow}$ & $K \left( \begin{smallmatrix} 1 & 2 \\ 3 & \infty \end{smallmatrix} \right)$  \\
&$\searrow$&$+$&$\searrow$&$-$&$\searrow$&$+$& $\searrow$ & $-$
\\
&&$K \left( \begin{smallmatrix} 1 & 2 \\ 3 & 1 \end{smallmatrix} \right)$&&$K \left( \begin{smallmatrix} 1 & 2 \\ 3 & 2 \end{smallmatrix} \right)$&&$K \left( \begin{smallmatrix} 1 & 2 \\ 3 & 3 \end{smallmatrix} \right)$&\ \quad& $K \left( \begin{smallmatrix} 1 & 2 \\ 3 &  4 \end{smallmatrix} \right)$.   \\
\end{tabular}

\noindent For example, $K \left( \begin{smallmatrix} 1 & 2 \\ 3 & 2 \end{smallmatrix} \right)$ is a singular twin with triple points by applying triple-point modifications at $t_{11}$, $t_{12}$, $t_{21}$, and by replacing the triangles  $t_{13}$ and $t_{22}$ with the two triple points.  
\end{example} 

By a similar argument to that in  \cite[Lemma~3, p.289]{Ohyama1995}, it is elementary to see: 
\begin{lemma}\label{lem:Kformula}
Let $f_m$ be an invariant of order $m$ $(\le n)$  with respect to triple-point modifications.  
If $K$ and $L$ are $T_n$-similar,  then we have 
\begin{equation}\label{eq:Kformula}
f_m (K) = f_m (L) + \sum_{1 \le i_j \le \alpha (j), 1 \le j \le n}  \varepsilon_{1i_1}\varepsilon_{2i_2} \cdots \varepsilon_{ni_n}    f_m \left(  K \left( \begin{matrix} 1&2& \dots & n \\ i_1& i_2 & \dots & i_n \end{matrix} \right) \right).  
\end{equation}
\end{lemma}
By definition, $K \left( \begin{matrix} 1&2& \dots & n \\ i_1& i_2 & \dots & i_n \end{matrix} \right)$ contains $n$ triple points.   Then by (\ref{eq:Kformula}),  we have 
\begin{claim}\label{claim:SimToInv}
If $K$ and $L$ are $T_n$-similar and $f_m$ is an invariant of order $m$ $(\le n-1)$ with respect to triple-point modifications, then 
\[
f_m (K) = f_m (L).  
\] 
\end{claim}
If $C$ and $C'$ are  two rondles which differ by a twin $p \in \LCS_n (\overline{PTW_k})$, then $C$ and $C'$ are $T_n$-similar.    
Therefore Claim~\ref{claim:SimToInv} applies, which proves  Theorem~\ref{thm:mainLCS}.  
$\hfill\Box$

\section{Infinitely many curves that match up to a given order}\label{sec:construction}
{\renewcommand{\theenumi}{\roman{enumi}}
In this section, we prove Theorem~\ref{thm:mainInf}. 
Theorem~\ref{thm:mainLCS} reduces the proof of  Theorem~\ref{thm:mainInf} to finding examples.  
By Theorem~\ref{thm:mainLCS}, letting $C_i = C^{(0)}_i$, it suffices to construct doubly indexed rondles $C_i, C^{(1)}_i, C^{(2)}_i, \cdots$ ($i=1, 2, 3, \dots$) satisfying Conditions~(\ref{Condi:PureTwin}),  (\ref{Condi:Different}), and (\ref{Condi:Index}) as follows: 
\begin{enumerate}
\item $C_i$ and $C^{(n)}_i$ differ by a twin $p \in \LCS_n (\overline{PTW_k})$. \label{Condi:PureTwin} 
\item $C^{(m)}_i \neq C^{(n)}_i$ ($m < n$) as rondles.  
 \label{Condi:Different}
\item $C^{(n)}_i \neq C^{(n)}_j$ ($i \neq j$) as rondles.\label{Condi:Index}
\end{enumerate}
}

Firstly, we consider $k=4$.  

Let $a$ and $b$ be twins as in Figure~\ref{fig:GeneFifthTri}, and let $\langle a, b \rangle$ ($\subset PTW_4$) be the subgroup generated by $a^{\pm 1}$ and $b^{\pm 1}$.  
By definition, it is an example of a small pure twin group (Definition~\ref{def:smallpuretwin}) since each of $a$, $a^{-1}$, $b$, and $b^{-1}$ reduces to $1$ by a single application of (\ref{eq:YB}) exactly once, with relations (\ref{eq:in}) and (\ref{eq:slide}) freely used. 
Hence the construction in Figure~\ref{fig:example} gives $C_i$ and $C^{(n)}_i$ which differ by $p \in \LCS_n(\langle a,b\rangle)$.  
Thus     
Condition~(\ref{Condi:PureTwin}) holds.      
In particular, using  generators $a^{\pm 1}, b^{\pm 1} \in \overline{PTW_4}$ as in Figure~\ref{fig:GeneFifthTri}, we have a pure twin $p \in \LCS_n (\overline{PTW_4})$ running over $
b, [a, b], [a, [a, b]], [a,[a, [a, b]]], \cdots 
$ (Figure~\ref{fig:GeneFifthTri}), which gives the corresponding rondles $C_i, C^{(1)}_i, C^{(2)}_i, C^{(3)}_i, \dots$ in the statement of Theorem~\ref{thm:mainInf} (we will later prove that they satisfy Condition~(\ref{Condi:Different})).  
Note that by Theorem~\ref{thm:mainLCS}, if $C_i$ ($i=1, 2, 3, \dots$) is a twin closure  
and $C^{(n)}_i$ includes $p \in \LCS_n (\overline{PTW_4})$ as in Figure~\ref{fig:example}, 
$f(C_i)=f(C^{(n)}_i)$ for any invariant $f$ of order less than $n$ with respect to triple-point modifications.       
\begin{figure}[htbp] 
   \centering
   \includegraphics[width=14cm]{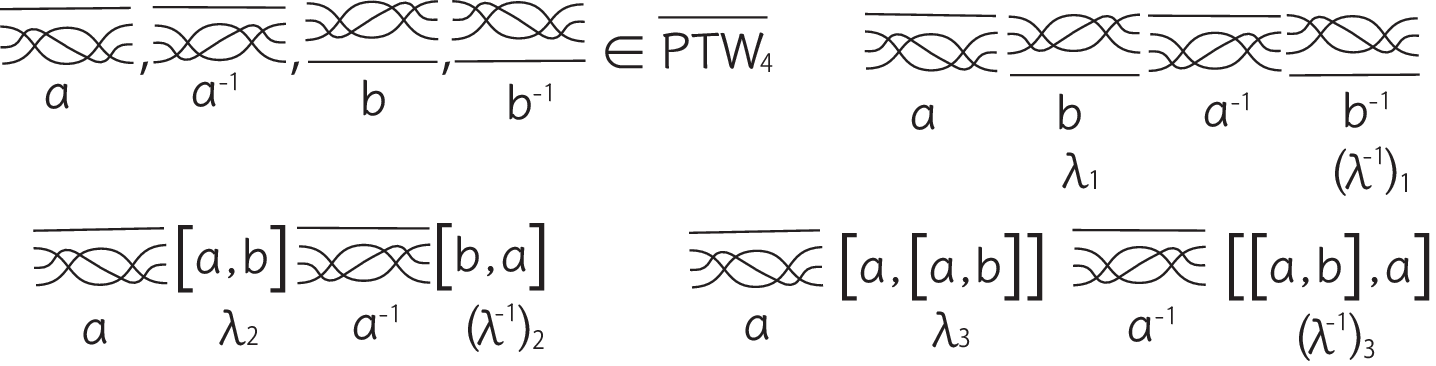} 
   \caption{Generators $a^{\pm}$, $b^{\pm}$ and   
 the sequence $\lambda_1=b$, $\lambda_2=[a, b]$, $\lambda_3=[a, [a,b]]$, \dots, $\lambda_n=[a \cdots [a, [a, b]]\cdots]$.  }
   \label{fig:GeneFifthTri}
\end{figure}
\begin{figure}[htbp] 
   \centering
   \includegraphics[width=12cm]{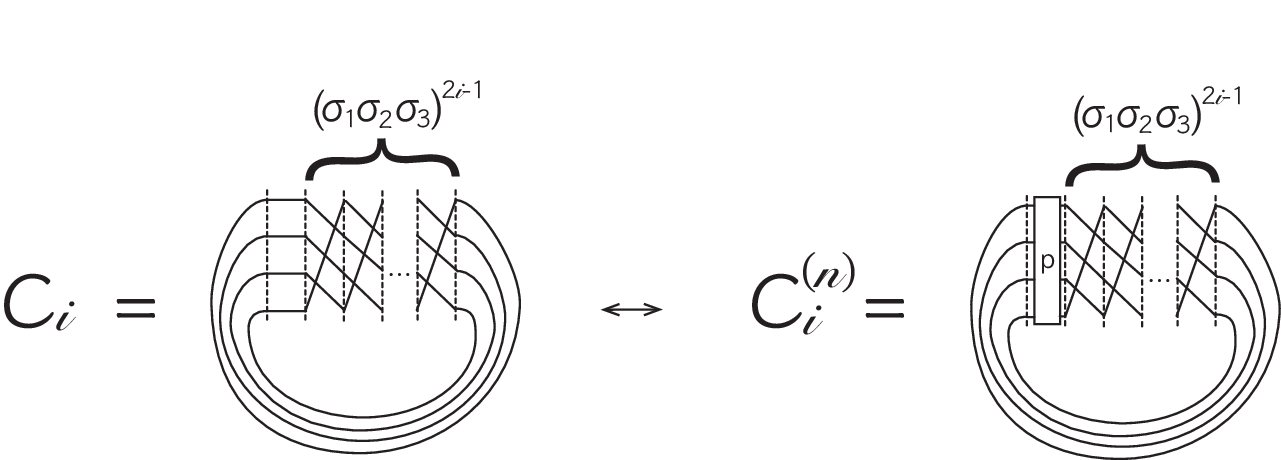} 
   \caption{Construction of prime curves from $TW_4$}
   \label{fig:example}
\end{figure}

Next, we show that the sequence  
 $C_i, C^{(1)}_i, C^{(2)}_i, C^{(3)}_i, \dots$ satisfies  Condition~(\ref{Condi:Different}), and that any pair $C^{(n)}_i, C^{(n)}_j$ satisfies Condition~(\ref{Condi:Index}).     
We prepare Notation~\ref{not:reduced} and recall Fact~\ref{fact:reduced}.  
\begin{notation}\label{not:reduced}
Following \cite{Khovanov1996}, a self-tangency modification \selftanMb decreasing double points is called a ``$(-2)$ move''.  Let $D$ be a representative of a rondle $C$ and let $D^r$ be a representative with no bigons where it is obtained from $D$ by successively applying only $(-2)$ moves.    
\end{notation}
\begin{fact}[Khovanov \cite{Khovanov1996}]\label{fact:reduced}
If $D$ and $D'$ are two representatives of a rondle $C$, then $D^r$ and ${D'}^r$ are identified up to diffeomorphisms of $\mathbb{R}^2$ with compact support.  In other words, $D^r$ is uniquely determined.   
\end{fact}
\begin{corollary}\label{cor:reduced}
Let $|D|$ be the number of double points in a representative of a rondle, and let $D^r$ be as in Fact~\ref{fact:reduced}.   If $|D^r| \neq |{D'}^r|$, then $D^r$ and ${D'}^r$ correspond to different rondles.   
\end{corollary}
\begin{example}\label{eg:CiCj}
In Figure~\ref{fig:example}, for any $i$, since no bigon appears, 
$C_i \neq C_j$ ($i \neq j$) as rondles by Corollary~\ref{cor:reduced}.   This proves the special case ($C^{(0)}_i \neq C^{(0)}_j$ for $i \neq j$) of Condition~(\ref{Condi:Index}).  
\end{example}  
The proof of \cite[Theorem~2.2]{Khovanov1996} implies the extension of Fact~\ref{fact:reduced} to configurations/twins (Definition~\ref{not:reducedTwin} and Theorem~\ref{thm:reducedTwin}).    
\begin{definition}\label{not:reducedTwin}
A representative of a configuration/twin is called \emph{bigon-free} if any bigon it contains is adjacent to the exterior region.  
For any configuration/twin $p$, a bigon-free representative is obtained from $p$ by successively applying $(-2)$ moves and it is called the \emph{bigon-free twin representative}.   For any twin, the symbol $D^r$ denotes a bigon-free twin representative.     
\end{definition}
By Theorem~\ref{thm:reducedTwin}, Definition~\ref{not:reducedTwin} is well-defined. 
\begin{theorem}\label{thm:reducedTwin}
If $D$ and $D'$ are two representatives of a configuration/twin, then $D^r$ and ${D'}^r$ are identified up to diffeomorphisms of $\mathbb{R}^2$ fixing the ends of arcs.  In other words, $D^r$ is uniquely determined from a given configuration/twin.  
\end{theorem}
The proof of Theorem~\ref{thm:reducedTwin} will be given in Section~\ref{sec:appendix}.

By Fact~\ref{fact:reduced}, we will check when a bigon appears in the process of repeatedly taking  commutators.   
Let $\{\lambda_n\}$ and $\{(\lambda^{-1})_n\}$ be sequences of commutators defined by 
\[
\lambda_1 = b, \lambda_{n+1} = [a, \lambda_{n}].\]
It is elementary to see the following: 
\begin{itemize}
\item $\lambda_1 = b$.  
\item $\lambda_2$ starts from the generator $a$ and ends with $b^{-1}$.  
\item $\lambda_n = [a, \lambda_{n-1}]$ ($n \ge 3$) starts from $a$ and ends with $a^{-1}$.  
      
\item $\lambda^{-1}_n$ ($n \ge 3$) starts from  $a$ and ends with $a^{-1}$.  
\item $a^2$ and $(a^{-1})^2$ produce no new bigons.   
\end{itemize}
Therefore, it suffices to determine where new bigons can appear in $\lambda_2$ and $\lambda_3$ by taking commutators.   
By Figure~\ref{fig:GeneFifthTri}, 
the only place where new bigons can appear is in $\lambda_3$,
namely in the subword $a^{-1}b$.
However, in $C_i$ of Figure~\ref{fig:example}, the (unique) bigon-free twin representative  obtained from $a^{-1} b$ does not 
change the boundary strands 
as in Figure~\ref{fig:aInvb}, and hence does not affect the terminal generators $a^{\pm 1}$ appearing  in $[a, \lambda_n]$ for all $n$.  This proves Condition~(\ref{Condi:Different}).  
Hence we have the sequence of  distinct twins $\lambda_1$, $\lambda_2$, $\lambda_3$, \dots.  
Therefore, we have: 
\begin{example}\label{eg:CnCn}
For non-trivial twin $p=\lambda_n$ ($n=1, 2, 3, \dots$) in Figure~\ref{fig:example} and for certain twins $w$ and  $w'$ starting with $a$ or $b$,    
since taking the closure of $p \sigma_1 \sigma_2 \sigma_3$ $=$ $w b^{\pm 1} \sigma_1 \sigma_2 \sigma_3$ or $w' a^{-1}  \sigma_1 \sigma_2 \sigma_3$ does not produce more new bigons,  
$C^{(n)}_i \neq C^{(n)}_j$ ($i \neq j$) as rondles by Corollary~\ref{cor:reduced}.  
\end{example} 
Thus Condition~(\ref{Condi:Index}) holds, which completes the proof 
for the case $k=4$.  
\begin{figure}[htbp] 
   \centering
   \includegraphics[width=12cm]{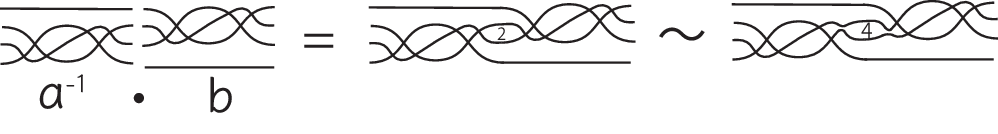} 
   \caption{Reduction of the twin $a^{-1} b$ with a bigon to its bigon-free twin representative, where the label $\ell$ ($=2$ or $4$) indicates an $\ell$-gon.  }
   \label{fig:aInvb}
\end{figure}


Secondly, we treat the general case $k>4$ by applying the argument for the case $k=4$ to the case $k >4$.  
In the following, let $C_{i, k}=C^{(0)}_{i, k}$, and let $C^{(n)}_{i, k}$ ($n=0, 1, 2, \dots$) denote  $C^{(n)}_i = T(p)$ such that $p \in PTW_k$ $(k \ge 4)$.   We will show Conditions (\ref{Condi:PureTwin}), (\ref{Condi:Different}), and (\ref{Condi:Index}).  
By adding straight lines to the generators $a^{\pm 1}, b^{\pm 1}$, 
define twins $a_k$ and $b_k$ as in Figure~\ref{fig:GeneK7} (left), and let $\langle a_k, b_k \rangle \subset PTW_k$ be the subgroup generated by $a_k^{\pm 1}$ and $b_k^{\pm 1}$.  
By definition, this is an example of a small pure twin group (Definition~\ref{def:smallpuretwin}) since each element $a_k$, $a_k^{-1}$, $b_k$, or $b_k^{-1}$ reduces to $1$ by applying (\ref{eq:YB}) exactly once, with relations (\ref{eq:in}) and (\ref{eq:slide}) freely used. 
Hence the construction in Figure~\ref{fig:GeneK7} (right) gives $C_{i, k}$ and $C^{(n)}_{i, k}$ which differ by $p \in \LCS_n (\langle a_k, b_k \rangle)$.     
Thus Condition~(\ref{Condi:PureTwin}) holds.    
In particular, using generators $a_k^{\pm 1}, b_k^{\pm 1} \in \overline{PTW_k}$ as in Figure~\ref{fig:GeneK7} (left), we obtain a pure twin $p \in \LCS_n (\overline{PTW_k})$ running  over $b_k, [a_k, b_k], [a_k, [a_k, b_k]], [a_k,[a_k, [a_k, b_k]]] \cdots$, which gives the corresponding rondles $C_{i, k}, C^{(1)}_{i, k}, C^{(2)}_{i, k}, C^{(3)}_{i, k}, \dots$ in the statement of Theorem~\ref{thm:mainInf} (we will later prove that they satisfy Condition~(\ref{Condi:Different})).   
Note that by Theorem~\ref{thm:mainLCS}, 
if $C_{i, k}$ ($i=1, 2, 3, \dots$) is a twin closure and $C^{(n)}_{i, k}$ includes $p_k \in \LCS_n (\overline{PTW_k})$ as in Figure~\ref{fig:GeneK7}, 
$f(C_{i, k})=f(C^{(n)}_{i, k})$ for any invariant $f$ of order less than $n$ with respect to triple-point modifications.       
\begin{figure}[htbp] 
   \centering
   \includegraphics[width=14.5cm]{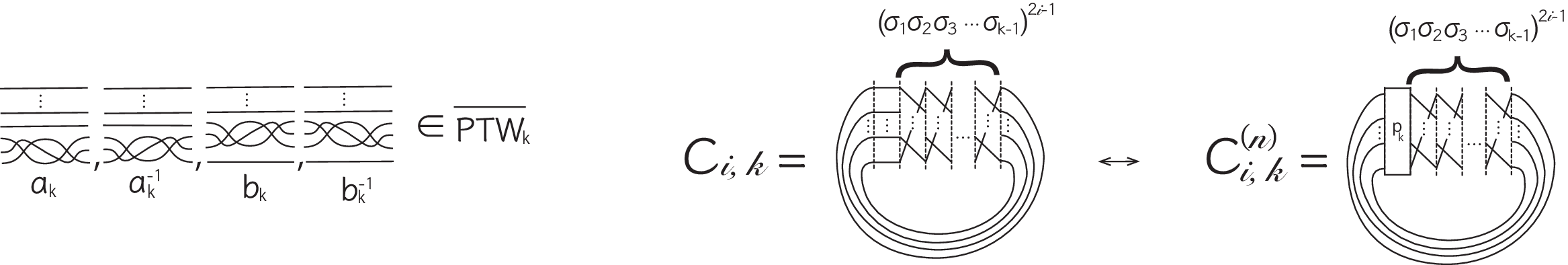} 
   \caption{Generators (left) and $C_{i, k}$ and $C^{(n)}_{i, k}$ (right)}
  \label{fig:GeneK7}
\end{figure}

Next, we show that the sequence  $C_{i, k}$, $C^{(1)}_{i, k}$, $C^{(2)}_{i, k}, C^{(3)}_{i, k}, \dots$ satisfies  Condition~(\ref{Condi:Different}), and that any pair $C_{i, k}, C_{j, k}$ satisfies  Condition~(\ref{Condi:Index}).     
\begin{example}\label{eg:CiCjk}
In Figure~\ref{fig:GeneK7} (right), for any $i$, since no bigon appears, 
$C_{i, k} \neq C_{j, k}$ ($i \neq j$) as rondles, by Corollary~\ref{cor:reduced}.   This proves the special case ($C^{(0)}_{i, k} \neq C^{(0)}_{j, k}$ for $i \neq j$) of Condition~(\ref{Condi:Index}).
\end{example}

By Fact~\ref{fact:reduced}, we will check when a bigon appears in the process of repeatedly taking  commutators.   
Let $\{\lambda_{n, k}\}$ and $\{(\lambda^{-1})_{n, k}\}$ be sequences of commutators defined by 
\[
\lambda_{1, k} = b_k, \lambda_{n+1, k} = [a_k, \lambda_{n, k}].\]
It is elementary to see the following: 
\begin{itemize}
\item $\lambda_{1, k} =b_k$.  
\item $\lambda_{2, k}$ starts from the generator $a_k$ and ends with $b_k^{-1}$.  
\item $\lambda_{n, k} = [a_k, \lambda_{n-1, k}]$ ($n \ge 3$) starts from $a_k$ and ends with $a_k^{-1}$.         
\item $\lambda^{-1}_{n, k}$ ($n \ge 3$) starts from $a_k$ and ends with $a_k^{-1}$.
\item $a_k^2$ and $(a_k^{-1})^2$ produce no new bigons.   
\end{itemize}
Therefore, it suffices to determine where new bigons can appear in $\lambda_{2, k}$ and $\lambda_{3, k}$ by taking commutators.   By Figure~\ref{fig:GeneK7}, only possibility producing new bigons is on $\lambda_{3, k}$, in particular, $a_k^{-1} b_k$.  However, in $C_{i, k}$ of Figures~\ref{fig:example} and \ref{fig:GeneK7}, the (unique) bigon-free twin representative  obtained from $a_k^{-1} b_k$ does not change  the boundary strands by the same argument as  Figure~\ref{fig:aInvb}, and hence does not affect the terminal  generators $a_k^{\pm 1}$ appearing in $[a_k, \lambda_{n, k}]$ for all $n$.  This proves Condition~(\ref{Condi:Different}).  
Hence we have the sequence of  distinct twins $\lambda_{1, k}$, $\lambda_{2, k}$, $\lambda_{3, k}$, \dots.  
We thus obtain: 
\begin{example}\label{eg:CnCnk}
For non-trivial twin $p_k =\lambda_{n, k}$ ($n=1, 2, 3, \dots$) in Figure~\ref{fig:GeneK7} and for certain twins $w$ and  $w'$ starting with $a_k$ or $b_k$,    
since taking the closure of $p_k  \sigma_1 \sigma_2 \sigma_3 \cdots \sigma_{k-1}$ $=$ $w b^{\pm 1}_k \sigma_1 \sigma_2 \sigma_3 \cdots \sigma_{k-1}$ or $w' a^{-1}_k  \sigma_1 \sigma_2 \sigma_3 \cdots \sigma_{k-1}$ does not produce more new bigons,  
$C^{(n)}_{i, k} \neq C^{(n)}_{j, k}$ ($i \neq j$) as rondles by Corollary~\ref{cor:reduced}.  
\end{example} 
Thus Condition~(\ref{Condi:Index}) holds, which completes the proof for the case $k>4$. 
$\hfill\Box$

\begin{remark}
If the reader prefers a diagrammatic proof as in Section~\ref{sec:DiagrammaticProof} for Theorem~\ref{thm:mainLCS}, an example of taking $A_1, A_2, A_3, \dots$ is as in Figure~\ref{fig:GeneFifthTri} for the case $k=4$.  The generalization to $k \ge 4$ is straightforward.  The details are left to the reader.   
\begin{figure}[htbp] 
   \centering
   \includegraphics[width=12cm]{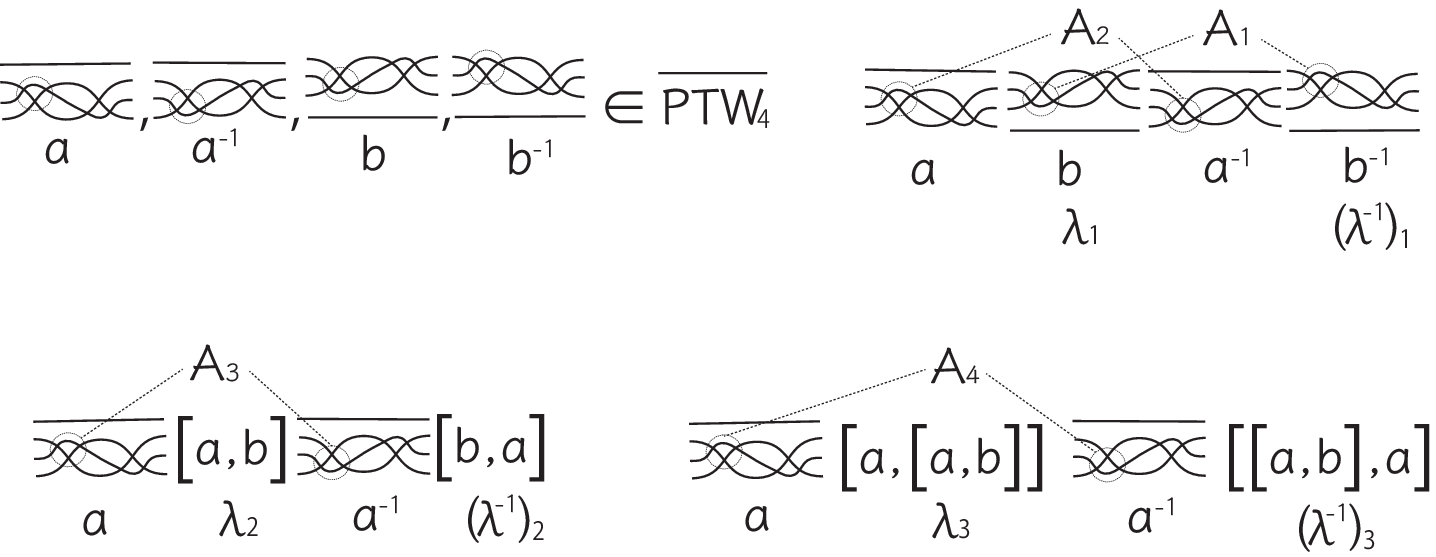} 
   \caption{An example of construction of $T_n$-trivial representatives of a twin having $A_1, A_2, A_3, \dots,$ or $A_n$.}
   \label{fig:GeneFifthTri}
\end{figure}  
\end{remark}

\section{The minimal representative of a twin (Proof of Theorem~\ref{thm:reducedTwin})}\label{sec:appendix}
In this section, we give a proof of Theorem~\ref{thm:reducedTwin}, 
following the proof of \cite[Theorem~2.2]{Khovanov1997} adapted to our setting
\footnote{
Comments for readers of \cite{Khovanov1997}. 
We use the symbol $D$ instead of $\Delta$ throughout this paper.
}.  

If $D', D''$ are representatives of the same configuration/twin, there is a sequence of representatives $D' = D_1, D_2, \dots, D_n = D''$ such that any two consecutive representatives are connected by one of $(\pm 2)$ moves.  
\begin{lemma}[{cf.~\cite[Lemma~2.1]{Khovanov1997}}]\label{lem:minimalDiag}
Let $p$ be a twin. Then for any two representatives 
$D', D''$ of $p$, there is a sequence of representatives
\[
D'=D_1, D_2, \dots, D_n=D''
\]
connected by $(\pm2)$ moves such that no $(+2)$ move
precedes a $(-2)$ move.
\end{lemma}
\noindent{\it Proof of Lemma~\ref{lem:minimalDiag}.}  Start with any sequence of representatives $D'=D_1, D_2, \dots, D_n =D''$, connecting $D'$ and $D''$.  Suppose that the move $m_i : D_i \to D_{i+1}$ is a $(+2)$ move and $m_{i+1} : D_{i+1} \to D_{i+2}$ is a $(-2)$ move.   If $m_{i+1}$ does not destroy at least one double point created by $m_i$, 
then we exchange the order of $m_i$ and $m_{i+1}$.
The other cases satisfy the condition that $m_i$ is a $(+2)$ move and $m_{i+1}$ is a $(-2)$ move.  These two moves $m_i$ and $m_{i+1}$ cancel.  
More precisely, we have exactly two cases as in Figure~\ref{fig:MinimalDiag} if $m_{i+1}$ destroys at least one double point $c$ created by $m_i$.  
Repeating this reduction and arguing by induction on the length $n$ of the sequence,  Lemma~\ref{lem:minimalDiag} is established.  $\hfill\Box$
\begin{figure}[htbp] 
   \centering
   \includegraphics[width=12cm]{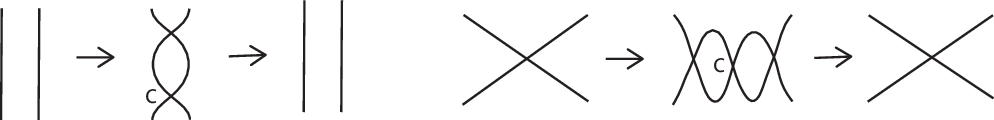} 
   \caption{Two cases satisfying that $m_{i+1}$ destroys at least one double point $c$ created by $m_i$.}
   \label{fig:MinimalDiag}
\end{figure}

\noindent{\it Proof of Theorem~\ref{thm:reducedTwin}.}
Lemma~\ref{lem:minimalDiag} implies that every twin $p$ admits a representative with the minimal number of double points, 
denoted by $D(p)$, and that such a representative is unique up to diffeomorphisms fixing the ends of arcs; that is, any other representative is obtained from $D(p)$ using only $(+2)$ moves.
$\hfill\Box$

\section*{Acknowledgements}
This work was supported by JSPS KAKENHI (Grant Numbers JP20K03604, JP22K03603, and JP25K06999).  
The author would like to thank Professor Mayuko Kon for discussions that contributed to the development of this work.

\bibliographystyle{plain}
\bibliography{CommSt}

\end{document}